\documentclass[12pt]{article}
\usepackage{amsmath,amssymb,amsthm} 
\usepackage[unicode,breaklinks=true,colorlinks=true]{hyperref}
\usepackage[dvipsnames]{xcolor}

\numberwithin{equation}{section}
\newtheorem{theorem}{Theorem}[section]

\newtheorem{lemma}[theorem]{Lemma}
\newtheorem{definition}[theorem]{Definition} 

\newtheorem{assumption}[theorem]{Assumption}

\theoremstyle{remark}
\newtheorem{remark}[theorem]{Remark}

\newcommand{\bke}[1]{\left( #1 \right)}

\newcommand{\bket}[1]{\left\{ #1 \right\}}
\newcommand{\norm}[1]{\| #1 \|}

\newcommand{\bka}[1]{\left\langle #1 \right\rangle}

\newcommand{\al}{\alpha}
\newcommand{\be}{\beta}

\newcommand{\e}{\epsilon}

\newcommand{\la}{\lambda}

\newcommand{\Om}{{\Omega}}

\newcommand{\si}{\sigma}

\newcommand{\De}{\Delta}

\newcommand{\Bp}{\dot B_{p,\infty}^{3/p-1}}
\newcommand{\Bq}{\dot B_{q,\infty}^{3/q-1}}
\newcommand{\bp}{\dot b_{p,\infty}^{3/p-1}}

\newcommand{\R}{{\mathbb R }}
\newcommand{\N}{{\mathbb N}}
\newcommand{\Z}{{\mathbb Z}}
\newcommand{\cR}{{\mathcal R}}

\newcommand{\nb}{{\nabla}}
\newcommand{\lec}{\lesssim}

\newcommand{\I}{\infty}

\newcommand{\supp}{\mathop{\mathrm{supp}}}

\newcommand{\donothing}[1]{{}}

\newcommand{\EQ}[1]{\begin{equation}\begin{split} #1 \end{split}\end{equation}}
\newcommand{\EQN}[1]{\begin{equation*}\begin{split} #1 \end{split}\end{equation*}}

\makeatletter
\newcommand{\xRightarrow}[2][]{\ext@arrow 0359\Rightarrowfill@{#1}{#2}}
\makeatother

\usepackage{accents}

\begin{document}

\title{Discretely self-similar solutions to the Navier-Stokes equations {with Besov space data}} 
\author{Zachary Bradshaw and Tai-Peng Tsai}
\date{\today}
\maketitle 

\begin{abstract}{We construct self-similar solutions to the three dimensional Navier-Stokes equations for divergence free, self-similar initial data that can be large in the critical Besov space $\Bp$ where $3< p< 6$.  We also construct discretely self-similar solutions for divergence free initial data in $\Bp$ for $3<p<6$ that is discretely self-similar for some scaling factor $\la>1$.  These results extend those of \cite{BT1} which dealt with initial data in $L^3_w$ since $L^3_w\subsetneq \Bp$ for $p>3$.  We also provide several concrete examples of vector fields in the relevant function spaces. }
\end{abstract}


\section{Introduction}

The three dimensional Navier-Stokes equations (3D NSE) are
\begin{equation} 
\begin{array}{ll}\label{eq:NSE}
 \partial_t v -\Delta v +v\cdot\nabla v+\nabla \pi  = 0&
\\  \nabla\cdot v = 0&
\end{array}
\mbox{~in~}\R^3\times [0,\infty).
\end{equation}
The velocity field evolves from a given initial data $v_0:\R^3\to \R^3$.  In 1934, Leray constructed weak (i.e.~distributional) solutions for initial data in $L^2$ in \cite{leray} and proved a priori bounds for his solutions. He also observed that any solution to \eqref{eq:NSE} has a natural scaling: if $v$ satisfies \eqref{eq:NSE}, then for any $\lambda>0$
\begin{equation}
	v^{\lambda}(x,t)=\lambda v(\lambda x,\lambda^2t),
\end{equation}
is also a solution with pressure 
\begin{equation}
	\pi^{\lambda}(x,t)=\lambda^2 \pi(\lambda x,\lambda^2t),
\end{equation}
and initial data 
\begin{equation}
v_0^{\lambda}(x)=\lambda v_0(\lambda x).
\end{equation}
A solution is called self-similar (SS) if $v^\lambda(x,t)=v(x,t)$ for all $\lambda>0$ and is discretely self-similar with factor $\lambda$ (i.e.~$v$ is $\lambda$-DSS) if this scaling invariance holds for a given $\lambda>1$. Similarly, $v_0$ is self-similar (a.k.a.~$(-1)$-homogeneous) if $v_0(x)=\lambda v_0(\lambda x)$ for all $\lambda>0$ or $\lambda$-DSS if this holds for a given $\lambda>1$.  
These solutions can be either forward or backward if they are defined on $\R^3\times (0,\I)$ or $\R^3\times (-\I,0)$ respectively.  In this paper we work exclusively with forward solutions.

Self-similar solutions satisfy an ansatz for $v$ in terms of a time-independent profile $u$, namely, 
\begin{equation}\label{ansatz1}
v(x,t) = \frac 1 {\sqrt {t}}\,u\bigg(\frac x {\sqrt{t}}\bigg), 
\end{equation} 
where $u$ solves the \emph{Leray equations}
\begin{equation} 
\begin{array}{ll}\label{eq:stationaryLeray}
 -\Delta u-\frac 1 2 u-\frac 1 2 y\cdot\nabla u +u\cdot \nabla u +\nabla p = 0&
\\  \nabla\cdot u=0&
\end{array}
\mbox{~in~}\R^3,
\end{equation}
in the variable $y=x/\sqrt{ t}$.
Discretely self-similar solutions are determined by their behavior on the time interval $1\leq t\leq \lambda^2$ and satisfy the ansatz
\begin{equation}\label{ansatz2}
v(x,t)=\frac 1 {\sqrt{t}}\, u(y,s),
\end{equation}
where
\begin{equation}\label{variables}
y=\frac x {\sqrt{t}},\quad s=\log t.
\end{equation}
The vector field $u$ is $T$-periodic with period $T=2\log \lambda$ and solves the \emph{time-dependent Leray equations}
\begin{equation} 
\begin{array}{ll}
\label{eq:timeDependentLeray}
 \partial_s u-\Delta u-\frac 1 2 u-\frac 1 2 y\cdot\nabla u +u\cdot \nabla u +\nabla p = 0& 
\\  \nabla\cdot u = 0&
\end{array}
\mbox{~in~}\R^3\times \R.
\end{equation}
Note that the \emph{similarity transform} \eqref{ansatz2}--\eqref{variables} gives a one-to-one correspondence between solutions to \eqref{eq:NSE} and \eqref{eq:timeDependentLeray}. Moreover, when $v_0$ is SS or DSS, the initial condition $v|_{t=0}=v_0$ corresponds to a boundary condition for $u$ at spatial infinity, see \cite{KT-SSHS,BT1,BT2}.

{Self-similar and discretely self-similar solutions are important since they might shed light on questions about blow-up and uniqueness.  Indeed, backward self-similar solutions were first introduced by Leray in \cite{leray} as candidates for singular solution. Ne\v cas, {R\accent23 u\v {z}i\v {c}ka} and {\v Sver\'ak} ruled out this possibility in \cite{NRS}, but the existence of nontrivial backward DSS solutions remains open.  {Forward} self-similar and discretely self-similar solutions are important as they are compelling candidates for non-uniqueness \cite{JiaSverak} and other, more technical properties \cite{BT1}.  Proving the existence of such solutions is the first step to pursuing these questions further.}  

Until recently, self-similar solutions were known to exist only for small data in scaling invariant function spaces such as $L^3_w,\, \Bp$ ($p<\I$), or $BMO^{-1}$ \cite{GiMi,Kato,CP,Barraza,Koch-Tataru}. 
The first large-data solutions were constructed by Jia and \v Sver\'ak in \cite{JiaSverak} and required the initial data to be H\"older continuous away from the origin.  Tsai adapted the approach of Jia and \v Sver\'ak to the discretely self-similar case in \cite{Tsai-DSSI}, and, in collaboration with Korobkov with a contradiction argument, to the case of self-similar solutions on the half-space \cite{KT-SSHS}. 
These large-data existence results all require the initial data is continuous away from the origin.  Bradshaw and Tsai eliminated this assumption in \cite{BT1} giving a construction for any SS/DSS data in $L^3_w$.  {Bradshaw and Tsai also treated a more general problem on the whole and half spaces in \cite{BT2} where they constructed \emph{rotated} self-similar and discretely self-similar solutions.}

On the whole space, the solutions of \cite{BT1,BT2} are in the local Leray class, which is a generalization of Leray's weak solutions that replaces global quantities with local analogues.   Lemari\'e-Rieusset introduced local Leray solutions in \cite[Chapters 32 and 33]{LR} and offered a construction.  Kikuchi and Seregin gave a revised construction with more details in \cite{KiSe}. 
Note that $L^3_w$ embeds in $L^2_{u\,loc}$, making it a natural place to seek self-similar solutions.  The main results of \cite{BT1} are the following two theorems.

\begin{theorem}\label{thrm.old2} {\normalfont \cite[Theorem 1.3]{BT1}}
Let $v_0$ be a $(-1)$-homogeneous divergence free vector field in $\R^3$ which satisfies
\begin{equation}\label{ineq:decayingdata}
\|v_0\|_{L^3_w(\R^3)}\leq c_0,
\end{equation} 
for a possibly large constant $c_0$. Then, there exists a local Leray solution $v$ to \eqref{eq:NSE} which is self-similar and additionally satisfies
\begin{equation}\label{thrm.old-conv}
 \|  v(t)-e^{t\Delta}v_0 \|_{L^2(\R^3)}\leq C_0\,t^{1/4}
\end{equation}
for any $t\in (0,\infty)$ and a constant $C_0=C_0(v_0)$.
\end{theorem}

\begin{theorem}\label{thrm.old} {\normalfont \cite[Theorem 1.2]{BT1} }
Let $v_0$ be a divergence free, $\lambda$-DSS vector field for some $\lambda >1$ and satisfy 
\eqref{ineq:decayingdata}
for a possibly large constant $c_0$. Then, there exists a local Leray solution $v$ to \eqref{eq:NSE} which is $\lambda$-DSS and additionally satisfies \eqref{thrm.old-conv}
for any $t\in (0,\infty)$ and a constant $C_0=C_0(v_0)$.
\end{theorem}

In his 2016 book \cite{LR2}, Lemari\'e-Rieusset provides a slightly more general result in the self-similar case by extending the Leray-Schauder approach of Jia and Sverak.  In particular, Lemari\'e-Rieusset first shows that any self-similar initial data in $L^\I(S^2)$ where $S^2$ denotes the unit sphere gives rise to a self-similar local Leray solution.  
He then shows that any self-similar initial data in $L^2_{u\,loc}$ can be approximated by self-similar initial data in $L^\I(S^2)$. Since all local Leray solutions satisfy an a priori bound, the constructed local Leray solutions can be used to approximate a self-similar solution for any data in $L^2_{u\,loc}$.  {We anticipate a similar argument can be made for discretely self-similar data and solutions generalizing Theorem \ref{thrm.old} to a larger class of initial data, and intend to elaborate on this in future research.}  Note that Chae and Wolf recently released a pre-print \cite{Chae-Wolf} which constructs solutions for DSS data in $L^2_{loc}(\R^3)$ via a different approach.

In this paper we generalize Theorems \ref{thrm.old2} and \ref{thrm.old} to cover self-similar and discretely self-similar data in the critical Besov spaces $\Bp$ where $3<  p< 6$, for any scaling factor $\la>1$.
In comparison to other well known spaces we have the following strict embeddings for $3<p<\infty$,
\[L^3_w\subset \Bp \subset BMO^{-1}\subset \dot B_{\infty,\infty}^{-1}.\]
{Note $ \Bq \subset \Bp$ if $3 \le q < p < \infty$.}
If $p=3$ then $L^3_w$ and $\Bp$ are not directly comparable.  

The following theorems are the main results of this paper.

\begin{theorem}\label{theorem.main2}Fix $p\in (3,6)$.  Assume $v_0:\R^3\to \R^3$ is divergence free, belongs to $\Bp$, and is self-similar. Then there exists a self-similar distributional solution $v$ and pressure distribution $\pi$ to 3D NSE on $\R^3\times (0,\infty)$.  
Furthermore, $v$ and $v_0$ can be decomposed as $a+b$ and $a_0+b_0$ respectively so that $a_0\in L^3_w$, $b_0$ is small in $\Bp$,  {
\EQ{\label{th1.3-1}
\|a(t)-e^{t\Delta}a_0\|_{L^2} &\leq C_2 t^{1/4},\\
\int_{0}^t \norm{a(\tau)-e^{\tau\Delta}a_0}_{L^r}^q d\tau  &\le C_r t^{q/4},\quad \forall\, r \in (2,6],
}
for some constant $C_r(v_0)$ with $\frac 3r + \frac 2q=\frac 32$, and  
\begin{equation}\label{th1.3-2}
\norm{b(t)-e^{t\Delta}b_0}_{L^r} \le C_r \norm{b_0}_{\Bp}^2 t^{-\frac 12 + \frac 3{2r}}, \quad \forall\, r \in 
\left[\frac p2, \frac {3p}{6-p}\right),
\end{equation}
for some constant $C_r$. }
Also, $a$ and $b$ are self-similar.
\end{theorem}

\begin{theorem}\label{theorem.main}Fix $p\in (3,6)$.  Assume $v_0:\R^3\to \R^3$ is divergence free, belongs to $\Bp$, and is $\la$-DSS for some $\la>1$. Then, there exists a $\la$-DSS distributional solution $v$ and pressure distribution $\pi$ to 3D NSE on $\R^3\times (0,\infty)$.  
Furthermore, $v$ and $v_0$ can be decomposed as $a+b$ and $a_0+b_0$ respectively so that $a_0\in L^3_w$, {$b_0$ is small in $\Bp$, $a(t)-e^{t\Delta}a_0$ satisfies \eqref{th1.3-1}, 
$b(t)-e^{t\Delta}b_0$ satisfies \eqref{th1.3-2}, and  }
$a$ and $b$ are $\la$-DSS. 
\end{theorem}

\noindent Comments on Theorems \ref{theorem.main2} and \ref{theorem.main}:

\begin{enumerate} 
\item If $p>3$, then there exist discretely self-similar functions in $\Bp\setminus L^3_w$, a fact we prove in Lemma \ref{lemma.strictembedding}.  
\item {The estimate \eqref{th1.3-1} is because $a(t)-e^{t\Delta}a_0$ is in the energy class in similarity variables. The estimate \eqref{th1.3-2} is a usual bilinear estimate for mild solutions. Combining both we have, 
for all {$r \in [\frac p2, 3)$}, $\frac 3r + \frac 2q=\frac 32$, 
\EQ{
\bke{\frac 1t \int_{0}^t \norm{v(\tau) - e^{\tau \De}v_0}_{L^r}^q d\tau}^{1/q} \le 
C   t^{-\frac 12 + \frac 3{2r}}.
}
Note the exponent on the right side is positive {for $r\in [\frac p2, 3)$}.
It shows that $v(t)$ converges to $e^{t\De}v_0$ as $t \to 0$ in some weak time-average sense, in a way that is independent of the decomposition $v_0=a_0+b_0$.}
\item In contrast to Theorems \ref{thrm.old2} and \ref{thrm.old}, we do not seek local Leray solutions since we do not have the embedding $\Bp \subset L^2_{\mathrm{loc}}$ for $p\geq 3$. Indeed, it is possible to show that there exist $2$-DSS initial data in $\Bp\setminus L^2_{\mathrm{loc}}$ -- see Lemma \ref{lemma.notsquareintegrable}.  This also ensures that our result is new in comparison to \cite[Theorem 16.3]{LR2} {and \cite{Chae-Wolf}.}   
\item In \cite{BT2} we proved the existence of solutions which were rotated self-similar and rotated discretely self-similar and had data in $L^3_w$.  The class of rotated discretely self-similar solutions includes but is larger than the DSS class.  Such solutions have an ansatz which satisfies a system resembling the stationary and time-periodic Leray equations and it is expected that, on the whole space, the arguments in this paper can be applied to construct rotated SS and rotated DSS solutions with data in $\Bp$ ($3<p<6$), but we do not include the details presently. 
\end{enumerate}

We prove Theorems \ref{theorem.main2} and \ref{theorem.main} similarly.  The idea is to decompose the initial data $v_0$ as $v_0=a_0+b_0$ where $a_0$ is large in $L^3_w$ and $b_0$ is small in $\Bp$. In the DSS case we use the Littlewood-Paley decomposition of $v_0$ (see Lemma \ref{lemma.profileslicing}) while in the self-similar case we use a lemma due to Cannone \cite[Proposition 23.1]{LR}.  {The small data $b_0$ gives rise to a SS/$\la$-DSS mild solution $b$ in the Kato space 
\EQ{
K_p  = \bket{ u \in C((0,\infty);L^p)\ : \ \norm{u}_{K_p} = \sup_{0<t<\infty} t^{\frac 12 - \frac 3{2p}}\norm{u(t)}_{L^p}},
}
see \cite[Theorem 5.27]{BCD}.}
We then construct a SS/$\la$-DSS solution $a$ to a perturbed problem by extending the arguments in \cite{BT1}.

Our approach breaks down for $p\geq 6$.  Basically, (small) strong solutions in $\Bp$ for $p\geq 6$ do not decay rapidly enough as $|x|\to \I$ for us to get \emph{a priori} bounds for solutions to the time-periodic, perturbed Leray equations in the energy class -- see inequality \eqref{ineq.sourcetermbound}.  It is conceivable that our general approach can be used for data in $\Bp$ for any $p\in (3,\I)$ if we work in a class larger than the energy class. But constructing time-periodic solutions in such a context has not been done, even for the Navier-Stokes equations.  The expansion $v_0=a_0+b_0$ fails in $BMO^{-1}$ because $BMO^{-1}$ is an $L^\I$ based space.  Consequently, we don't expect the arguments in this paper to extend to the case of self-similar or discretely self-similar data in $BMO^{-1}$.

This paper is arranged as follows.  In Section \ref{sec.technical} we study discrete self-similarity in Besov spaces and give the main technical lemma.  In Section \ref{sec.evolutionLeray} we prove the existence of solutions to a time periodic, 
 perturbed Leray equation.  Section \ref{sec.DSS} contains the proof of Theorem \ref{theorem.main} which depends on Sections \ref{sec.technical} and \ref{sec.evolutionLeray}. The self-similar case is covered in Section \ref{sec.SS}.  In Section \ref{Appendix} we analyze the relationships between the collections of DSS vector fields in various function spaces, for example we show $L^2_{loc}\cap $\,DSS and $\Bp\cap$\,DSS are not comparable.

\section{Discrete self-similarity in critical Besov spaces}\label{sec.technical}

We first recall the Littlewood-Paley characterization of Besov spaces. Fix an inverse length scale $\la>1$.  
Let $B_r$ denote the ball of radius $r$ centered at the origin in $\R^3$.  Fix a non-negative, radial cut-off function $\chi\in C_0^\infty(B_{1})$ so that $\chi(\xi)=1$ for all $\xi\in B_{1/\la}$. Let $\phi(\xi)=\chi(\lambda^{-1}\xi)-\chi(\xi)$ and $\phi_j(\xi)=\phi(\lambda^{-j}\xi)$.  For a vector field $u$ of tempered distribution, let $\Delta_j u=(\mathcal F^{-1}\phi_j)*u$ for $j\in \N_0 $ and $\Delta_{-1}=(\mathcal F^{-1}\chi)*u$. Then, $u$ can be written as\[u=\sum_{j\geq -1}\Delta_j u.\]
If $(\mathcal F^{-1}\chi(\la^{-j}\cdot))*u\to 0$ as $j\to -\infty$ in the space of tempered distributions, then for $j\in \Z$ we define $\dot \Delta_j u = \mathcal F^{-1}\phi_j*u$ and have
\[u=\sum_{j\in \Z}\dot \Delta_j u.\] 
For $s\in \R$, $1\leq p,q\leq \infty$, the non-homogeneous Besov spaces include tempered distributions modulo polynomials for which the norm
\begin{align*}
&\|u\|_{B^s_{p,q}}:= 
\begin{cases} 
 \bigg(\sum_{ j\geq -1} \big(    \lambda^{sj} \|\Delta_j u \|_{L^p(\R^n)}  \big)^q \bigg)^{1/q}   & \text{ if } q<\infty  
\\ \sup_{j\geq -1} \lambda^{sj} \| \Delta_j u \|_{L^p(\R^n)} & \text{ if } q=\infty
\end{cases}, 
\end{align*}is finite, while the homogeneous Besov spaces include tempered distributions modulo polynomials for which the norm 
\begin{align*}
&\|u\|_{\dot B^s_{p,q}}:= 
\begin{cases} 
 \bigg(\sum_{ j\in \Z} \big(    \lambda^{sj} \|\dot \Delta_j u \|_{L^p(\R^n)}  \big)^q \bigg)^{1/q}   & \text{ if } q<\infty  
\\ \sup_{j\in \Z} \lambda^{sj} \| \dot \Delta_j u \|_{L^p(\R^n)} & \text{ if } q=\infty
\end{cases},
\end{align*}
is finite.  In this section we work with homogeneous Besov spaces while in \S \ref{sec.SS} we work with non-homogeneous spaces.

Besov spaces are typically defined using a \emph{dyadic} partition of unity in Fourier space -- i.e.~they are defined as above with $\la=2$.  If we are working with $\la$-DSS data, we want the partition of unity to be \emph{$\la$-adic}.  Fortunately, Besov spaces are independent of the scaling factor used to define the partition of unity on the Fourier side. 
To see this, let $\{\phi_j\}$ be a dyadic partition of unity satisfying the properties set forth at the beginning of this section and let $\{\phi_j^\la \}$ be a $\la$-adic partition of unity satisfying the same properties.
Let $\dot \De_j$ and $\dot \De_j^\la$ denote the homogeneous Littlewood-Paley operators generated by $\{\phi_j\}$ and $\{\phi_j^\la\}$ respectively. The next lemma confirms that $\dot \De_j$ and $\dot \De_j^\la$  generate equivalent norms for $\dot B^\si_{p,q}$ for any $\si\in \R$ and $1\leq q\leq \I$.  In particular, we have norm equivalence  for $\Bp$ when $3<p$.

\begin{lemma}
Let $\la>1$. Let $\dot \De_j$ and $\dot \De_j^\la$ be as defined above.  If $\si\in \R$ and $p,q \in [1, \infty]$, then any $f$ in the homogeneous Besov space $\dot B^\si_{p,q}$ satisfies  
\begin{align}
\label{Besov.lambda}
\norm{f}_{\dot B^\si_{p,q} } =\bigg\|   
	2^{\si k}\|\dot \De_k f \|_{L^p}
\bigg\|_{l^q_k}
\approx 
\bigg\|   
	\la^{\si j}\|\dot \De_j^\la f \|_{L^p}
\bigg\|_{l^q_j} .
\end{align}
\end{lemma}
\begin{proof}
Let $\phi(\xi)$ and $\phi^\la(\xi)$ be as above. So, $\phi$ is supported in $\frac 12 \le |\xi|\le 2$ and $\phi^\la$ is supported in $\la^{-1}\le |\xi|\le \la$.  Furthermore,
\[
\sum_j \phi (2^{-j} \xi) =\sum_j \phi^\la (\la^{-j} \xi) = 1\quad \forall \xi \not =0.
\]
Let $\dot \De_k f$ and $\dot \De_j^\la f$ be the corresponding Littlewood-Paley projection operators.
For each $j \in \Z$, let $S_j$ be the set of integers $k$ so that the intersection $[\la^{j-1},\la^{j+1}] \cap [2^{k-1},2^{k+1}]$ has positive measure. We have
$\dot\De_j ^\la f = \sum _{k \in S_j} \dot\De_j ^\la \dot\De_k f$, and thus
\EQ{
\la^{\si j} \norm{\dot \De_j ^\la f}_{L^p} 
&\le 
\sum_{k \in S_j} \la^{\si j} \norm{\dot\De_j ^\la \dot\De_k  f}_{L^p} 
\le 
\sum_{k \in S_j} \la^{\si j} C_1(\la)\norm{ \dot\De_k  f}_{L^p} 
}
Above we have used that $\dot\De_j ^\la$ is a convolution operator whose kernel is integrable with a uniform in $j$ bound $C_1(\la)$.  For each $k \in S_j$, we have
\[
\la^j = \la \la^{j-1} \le \la 2^{k+1} = 2\la 2^k,\quad
{\la^j = \la^{-1} \la^{j+1} \ge  \la^{-1} 2^{k-1} = (2\la)^{-1} 2^k.}
\]
Thus, {for all $\si \in \R$,}
\EQN{
\la^{\si j} \norm{\dot\De_j ^\la f}_{L^p} 
&\le C_1(\la)(2\la )^{|\si|}
\sum_{k \in S_j}  2^{k\si} \norm{ \dot\De_k  f}_{L^p} 
}
Since every $k \in S_j $ satisfies $(2\la)^{-1}\le 2^k \la ^{-j} \le 2 \la$, we have $\# S_j \le C(\la)$ independently of $j$. The above shows
\begin{align*}
\bigg\| \la^{\si j} \norm{\dot\De_j ^\la f}_{L^p} 
\bigg\|_{l^q_j}
&\leq
 C(\la,\si) \bigg\| \sum _{k\in S_j}  2^{k\si} \norm{\dot \De_k  f}_{L^p}\bigg\|_{l^q_j}
\\&\leq C(\la,\si)\bigg\|  2^{k\si} \norm{\dot \De_k  f}_{L^p}\bigg\|_{l^q_k},
\end{align*}
where $C(\la,\si)$ depends on $\la$ and $\si$ but not on $p$ or $q$.
The reversed inequality can be shown similarly. Hence we have \eqref{Besov.lambda}.
\end{proof}

The next lemma is the main technical result of this section.  It allows us to decompose any $\la$-DSS data in $\Bp$ into a small $\Bp$ part and a large $L^3_w$ part. The corresponding decomposition for \emph{self-similar} data is Lemma \ref{lemma.profileslicing2}.

\begin{lemma}\label{lemma.profileslicing}
Let $f$ be a $\la$-DSS, divergence free vector field in $\R^3$, and belong to $\dot B_{p,\infty}^{3/p-1}$ for some $\la \in (1,\infty)$ and $p\in (3,\infty)$. For any $\epsilon>0$, there exist divergence free $\la$-DSS distributions  $a\in L^3_w$ and $b\in \Bp$ so that 
 $f = a + b$ and $ \|b  \|_{\dot B_{p,\infty}^{3/p-1}}<\epsilon$.
\end{lemma}  

In the proof we will use the Helmholtz
projection $\mathbb P$ (or ``Leray projection'' in \cite[p.106]{LR}),
which maps a Banach space of vector fields in $\R^3$ to its subspace of  divergence free vector fields.
It is  
given by 
\begin{equation}
\label{Helmholtz}
(\mathbb Pg)_j= g_j+{\textstyle \sum_{k=1}^3} R_jR_k g_k
\end{equation}
where $R_k$ is the $k$-th Riesz transform with symbol $i \xi_k/|\xi|$.  In the variable $x$ this is given by the integral operator
\[
 R_kg ( x)=cP.V.\int \frac {y_k} {|y|^4} g(  x-y)\,dy.
\]
Note that  $\mathbb P$ is a bounded operator from $\Bp$ to $\Bp$ and from $L^3_w$ to $L^3_w$. For $\Bp$ spaces this is trivial since they're built on $L^p$ norms, where Calderon-Zygmund operators are bounded. For $L^3_w$, see \cite[Chapter 5,~Theorem 3.15]{Stein}.

\begin{proof} 
Let $f$ be as in the lemma's statement. {Let $\dot \De_j$ be the $\la$-adic spectral projection described in the beginning of this section.} Since $\dot \De_0 f\in L^p$, for any $\epsilon_1>0$, we may find functions $a_1$ and $b_1$ satisfying: 
\begin{align*}
&\dot \De_0 f = a_1 +b_1,
\\&  b_1\in L^p\mbox{ and } \|   b_1\|_{L^p}\leq \e_1,
\\& a_1\in C_0^\I.
\end{align*} 
Let 
\[
\tilde \De_0 = \sum_{j=-1,0,1}\dot \De_j.
\]
Looking at the Fourier side, it is clear that $\tilde \De_0 \dot \De_0  = \dot \De_0 $.
Let $a_2 =\tilde \De_0  a_1$ and $b_2 = \tilde \De_0 b_1$.  
Then, $\dot \De_0 f = \tilde \De_0 \dot \De_0 f= a_2+b_2$. 
Let 
\[
 	a = \mathcal F^{-1} \bigg(\sum_{j\in \Z}  \la^{2j}(\mathcal F a_2)(\la^j \xi)    \bigg),
\]
and 
\[
  b = \mathcal F^{-1} \bigg(\sum_{j\in \Z}  \la^{2j}(\mathcal F b_2)(\la^j \xi)    \bigg).
\]
Direct calculation shows that, if $f(x)$ is $\la$-DSS, that is, $f(x) = \la f(\la x)$ for any $x$, then its Fourier transform satisfies
\begin{align}\label{eq.interscale}
\hat f (\xi) =
\la^2 \hat f (\la \xi) ,\quad \forall \xi \in\R^3.
\end{align}It follows that
$\la^{2j}  \hat f(\la^j\xi) = \hat f(\xi)$ for any $j \in \mathbb{Z}$.
Thus,
\begin{align*}
(\hat a +\hat b )(\xi)&=\sum_{j\in \Z} \la^{2j} (\mathcal F(a_2+b_2))(\la^j \xi)
=\sum_{j\in \Z} \la^{2j} \big(\phi \hat f \big)(\la^j\xi) 
=\sum_{j\in \Z} \phi_{-j}(\xi)\hat f(\xi) = \hat f(\xi).
\end{align*} 
Therefore, $f=a+b$. By their construction, $a$ and $b$ satisfy \eqref{eq.interscale} and are therefore $\la$-DSS.  

Note that $f$ is $\la$-DSS if and only if 
\begin{align}\label{DSS.Fourier} 
\dot \Delta_j f(x)=\lambda^{j-i}\dot \Delta_i f(\lambda^{j-i} x), \quad \forall i,j\in \Z.
\end{align} 
This follows from the fact that
\[
f_j(x) \mapsto \phi_j(\xi) \hat f (\xi) =\la^{-2j} \phi(\xi\la^{-j})\hat f (\xi \la^{-j}) = \la^{-2j} (\phi \hat f) (\xi \la^{-j})\mapsto \la^{j}f_0(\la^jx),
\]
where $\mapsto$ is the image under either the Fourier or inverse Fourier transform and we have used the dilation property of the Fourier transform.

To obtain a bound for $b$ in $\Bp$, observe that
\begin{align*}
\mathcal F (\dot \De_0 b) =\phi_0(\xi) \sum_{j\in \Z}\la^{2j} ( \phi_{-1} +\phi_0 +\phi_1)(\la^j \xi) \mathcal F b_1(\la^j\xi). 
\end{align*}
Since $\phi_0(\xi) ( \phi_{-1} +\phi_0 +\phi_1)(\la^j \xi) =0 $ except for finitely many values of $j$, by Young's convolution inequality we have
\[
\|\dot \De_0 b\|_{L^p}\leq C \|b_1\|_{L^p} \leq C\epsilon_1,
\]
where $C$ only depends on our original choice of $\phi$.
 It follows from \eqref{DSS.Fourier} that
\[
\|b\|_{\Bp}\leq C \e_1.
\]

Since $a$ is $\la$-DSS, to show $a\in L^3_w$, it suffices to show $a\in L^\I (B_\la \setminus B_1)$.  Since $a_1\in C_0^\I$, we know $a_2$ is in the Schwartz class.  With a little work it follows that $\dot \De_0 a$ is also in the Schwartz class, and, therefore, $|\dot \De_0 a(x)|\lesssim (1+x^{2})^{-1}$.  Because $a$ is also $\la$-DSS, we see that
\[
 |a(x)|\leq \sum_{j\in \Z} |\dot \De_j a(x)|  \leq \sum_{j\in \Z} \la^j \dot{\De}_0 a(\la^jx) \lesssim \sum_{j\geq 0} \frac {\la^j} {1+\la^{2j}x^2} +\sum_{j<0} \la^j \|\dot \De_0a\|_{L^\I}<\I. 
\] 
Therefore, $a\in L^3_w$.

To make $a$ and $b $ divergence free we simply apply the Helmholtz
projection $\mathbb P$ \eqref{Helmholtz}.
With a slight abuse of notation, let $a=\mathbb P a$ and $b=\mathbb P b$ so that $a$ and $b$ are divergence free and we still have $f=a+b$.  Since $\mathbb P$ is a bounded operator on $\Bp$ and on $L^3_w$, we have 
 $a\in L^3_w$ and $b\in \Bp$.  Furthermore, by taking $\e_1$ sufficiently small we can ensure that $\|b\|_{\Bp}<\e$, where $\e$ is given in the lemma's statement.  

It remains to check that $\mathbb P$ preserves discrete self-similarity. 
If $g$ is $\la$-DSS for some $\la>1$ then
\begin{align*}
  R_kg ( x)=cP.V.\int \frac {y_k} {|y|^4} g(  x-y)\,dy
 &= \lambda cP.V. \int \frac {  y_k } {|  y  |^4}  g( \lambda x- \lambda y)  \,dy
\\& = \lambda cP.V.\int \frac {  \lambda y_k } {| \lambda y  |^4}  g( \lambda x- \lambda y)  \lambda^3 \,dy
\\&=\lambda cP.V.\int \frac {   z_k } {|   z  |^4}  g( \lambda x- z)  \,dz
\\&=\lambda R_kg(\lambda x),
\end{align*}
i.e.~$R_kg$ is also $\lambda$-DSS.  Hence $ a$ and $ b$ are discretely self-similar.
\end{proof}

\section{The time-periodic perturbed Leray equations}\label{sec.evolutionLeray}

In this section we construct a periodic weak solution to the perturbed Leray system
\begin{equation}
\label{eq:wholeSpaceLeray}
\begin{array}{ll}
	\partial_s u -\Delta u=\frac 1 2 u+\frac 1 2 y\cdot \nabla u -\nabla p -u\cdot\nabla u-B\cdot\nabla u - u\cdot\nabla B	&\mbox{~in~}\R^3\times \R
	\\  \nabla\cdot u = 0  &\mbox{~in~}\R^3\times \R
	\\ 	\displaystyle \lim_{|y_0|\to\infty} \int_{B_1(y_0)}|u(y,s)-U_0(y,s)|^2\,dx= 0& \mbox{~for all~}s\in \R
	\\  u(\cdot,s)=u(\cdot, s+T) &\mbox{~in~}\R^3\mbox{~for all~}s\in \R,
\end{array}
\end{equation}
for given $T$-periodic divergence free vector fields $B$ and $U_0$.  Here $U_0$ serves as the boundary value of the system and is required to satisfy the following assumption.
\begin{assumption} \label{AU_0}
The vector field $U_0(y,s) :\R^3 \times \R \to \R^3$ is continuously differentiable in $y$ and $s$,  periodic in $s$ with period $T>0$, divergence free, and satisfies
\begin{align*}
& 	\partial_s U_0-\Delta U_0-\frac 1 2 U_0-\frac 1 2 y\cdot \nabla U_0 = 0, 
\\& U_0\in L^\infty (0,T;L^4\cap L^q(\R^3)),
\\& \partial_s U_0\in L^\infty(0,T;L_{\mathrm{loc}}^{6/5}(\R^3)),
\end{align*}
and
\[
\sup_{s\in [0,T]}\|U_0  \|_{L^q(\R^3\setminus B_R)}\leq \Theta(R),
\]
for some $q\in (3,\infty]$ and $\Theta:[0,\infty)\to [0,\infty)$ such that $\Theta(R)\to 0$ as $R\to\infty$.
\end{assumption}

We seek solutions in the distributional sense where we are testing against test functions in $\mathcal D_T$, the collection of all smooth divergence free vector fields in $\R^3 \times \R$ which 
are time periodic with period $T$ and whose supports are compact in space.

\begin{definition}[Periodic weak solution]
\label{def:periodicweaksolutionR3} 
Let $U_0$ satisfy Assumption \ref{AU_0} and assume $B$ is $T$-periodic and divergence free. 
The field $u$ is a periodic weak solution to \eqref{eq:wholeSpaceLeray} in $\R^3\times (0,T)$ if it is divergence free, if 
\begin{equation}\notag
U:= 
u-U_0\in L^\infty(0,T;L^2(\R^3))\cap L^2(0,T;H^1(\R^3)),
 \end{equation} 
and if
\begin{equation}\label{u.eq-weak}
\int_0^T \bigg( (u,\partial_s f)-(\nabla u,\nabla f)+(\frac 1 2 u+\frac 1 2 y\cdot\nabla u-u\cdot\nabla u -u\cdot \nabla B -B\cdot\nabla u,f)  \bigg)  \,ds =0,
\end{equation} 
holds for all $f \in \mathcal D_T$.  
This latter condition implies that $u(0)=u(T)$.

\end{definition}

If $u$ satisfies this definition then there exists a pressure $p$ so that $(u,p)$ constitute a distributional solution to \eqref{eq:wholeSpaceLeray} (see the standard construction of $p$ in \cite{Temam}).    Our main existence theorem is the following.  
 
\begin{theorem}[Existence of solutions to \eqref{eq:wholeSpaceLeray}]\label{thrm:existenceOnR3}
Assume $U_0(y,s)$ satisfies Assumption \ref{AU_0} with $q=10/3$ and $B\in C^1(\R^4)\cap L^\infty(\R;L^p(\R^3))$ and satisfies $\|B\|_{L^\infty(\R^3\times (0,T))}<\frac 1 {24}$. Then \eqref{eq:wholeSpaceLeray} has a periodic  weak solution $u$ in $\R^4$ with period $T$.
\end{theorem}

To prove Theorem \ref{thrm:existenceOnR3} we replace $U_0$ by an auxiliary vector field $W$ which is constructed to ensure 
\[
\int (f\cdot\nabla W )\cdot f \leq \al \| f \|_{H^1}^2,
\]
for a given value $\al\in (0,1)$ and any $f\in H^1_0$.  This bound does not hold for general $U_0$ satisfying Assumption \ref{AU_0}. A suitable construction of $W$ is given in \cite[Lemma 2.5]{BT1} and we recall it for convenience. To do so, fix $Z\in C^\infty(\R^3)$ with $0 \le Z \le 1$, $Z(x)=1$ for $|x|>2$ and $Z(x)=0$ for $|x|<1$. This can be done so that $|\nb Z|+|\nb^2 Z| \lec 1$.
 For a given $R>0$, let $\xi(y)=Z(\frac yR)$. It follows that $|\nabla^k \xi|\lesssim R^{-k}$ for $k\in \{ 0,1\}$. 

\begin{lemma}[Revised asymptotic profile]
\label{lemma:W}
Fix $q\in (3,\infty]$ and suppose $U_0$ satisfies Assumption \ref{AU_0} for this $q$. 
Let $Z\in C^\infty(\R^3)$ be as above.
For any $\alpha\in (0,1)$, there exists $R_0=R_0(U_0,\alpha)\ge 1$ so that letting $\xi(y) =Z(\frac y{R_0})$ and setting
\begin{equation}
  W (y,s)= \xi(y) U_0(y,s) + w(y,s),
\end{equation}
where 
\begin{equation}
w(y,s)=\int_{\R^3}\nabla_y \frac 1 {4\pi |y-z|} \nabla_z \xi(z) \cdot U_0 (z,s) \,dz,
\end{equation}
we have that $W$ is locally continuously differentiable in $y$ and $s$, $T$-periodic, divergence free,
 $U_0 - W \in L^\infty(0,T; L^2(\R^3))$, and 
\begin{equation}\label{ineq:Wsmall}
\|W\|_{L^\infty(0,T;L^q(\R^3))}\leq \alpha, %
\end{equation} 
\begin{equation}\label{WL4.est}
\norm{W}_{L^\infty(0,T;L^4(\R^3))}\leq c(R_0,U_0),
\end{equation}
and
\begin{equation}
\label{LW.est}
\norm{	\partial_s W-\Delta W-\frac 1 2 W-\frac 1 2 y\cdot \nabla W}_{L^\infty(0,T; H^{-1}(\R^3))} \leq c(R_0,U_0), %
\end{equation}
where $c(R_0,U_0)$ depends on $R_0$ and quantities associated with $U_0$ which are finite by Assumption \ref{AU_0}.  
\end{lemma}

The proof of Lemma \ref{lemma:W} says more about $w$ (see \cite[Proof of Lemma 2.5]{BT1}).  In particular, since 
\begin{equation}
\label{ineq:wgradient}|\nabla w(y)|\leq \frac {C(R_0,U_0)} {1+|y|^{3}},
\end{equation}
we have $\nabla w\in L^2( 0,T;L^2(\R^3))$.

\begin{proof}[Proof of Theorem \ref{thrm:existenceOnR3}]
The argument is similar to that from \cite[Section 2]{BT1}. Fix $T>0$ and assume $U_0$ satisfies Assumption \ref{AU_0} for this $T$. Assume $B$ is a given $T$-periodic divergence free vector field.  Let $W$ be as defined in Lemma \ref{lemma:W} with $\al=\frac 1{24}$, $q=10/3$, and the given $U_0$. We look for a solution $u$ to \eqref{eq:wholeSpaceLeray} of the form $u=U+W$ where $U$ is divergence free and solves the perturbed system
\[
  \label{perturbed-Leray}
\partial_s U -\Delta U-\frac 1 2 U-\frac 1 2 y\cdot\nabla U + (W+U)\cdot \nabla U + U\cdot \nabla W +B\cdot \nabla U +U\cdot\nabla B +\nabla p = -  \mathcal R(W) ,
\] 
where the source term is
\[ \label{RW.def}
\mathcal{R}(W) := 	\partial_s W-\Delta W-\frac 1 2 W-\frac 1 2 y\cdot \nabla W + W\cdot\nabla W+B\cdot\nabla W+W\cdot\nabla B.
\]

We use the Galerkin method following \cite{GS06} (see also \cite{Temam}).  The relevant function spaces are
\begin{align*}
&\mathcal V=\{f\in C_0^\infty({ \R^3;\R^3}) ,\, \nabla \cdot f=0 \},
\\& X = \mbox{the closure of~$\mathcal V$~in~$H_0^1(\R^3)$} ,
\\& H = \mbox{the closure of~$\mathcal V$~in~$L^2(\R^3)$},
\end{align*}where $H_0^1(\R^3)$ is the closure of $C_0^\infty(\R^3)$ in the Sobolev space $H^1(\R^3)$.  Let $X^*(\R^3)$ denote the dual space of $X(\R^3)$. 
Let $(\cdot,\cdot)$ be the $L^2(\R^3)$ inner product and $\langle\cdot,\cdot\rangle$ be the dual product for $H^1$ and its dual space $H^{-1}$, or that for $X$ and $X^*$.
Let $\{a_{k}\}_{k\in \N}\subset \mathcal V$ be an orthonormal basis of $H$.
For a fixed $k$, we look for an approximation solution of the form $U_k(y,s)= \sum_{i=1}^k b_{ki}(s)a_i(y)$.
Here, $b_k=(b_{k1},\ldots,b_{kk})$  is a T-periodic solution to the system of ODEs
\begin{align}\label{eq:ODE}
\frac d {ds} b_{kj} = & \sum_{i=1}^k A_{ij}b_{ki} +\sum_{i,l=1}^k B_{ilj} b_{ki}b_{kl} +C_j,%
\end{align}
for $j\in \{1,\ldots,k\}$ and
\begin{align}
\notag A_{ij}&=- (\nabla a_{i},\nabla a_j) 
		+ \frac 1 2 (a_i+y\cdot \nabla a_i, a_j) 
		 -(    a_i\cdot \nabla( W+B ),a_j)
		- ((W+B)\cdot\nabla a_i, a_j)
\\\notag B_{ilj}&=- ( a_i \cdot\nabla a_l, a_j)
\\\notag C_j&=-\langle \mathcal R (W),a_j\rangle.
\end{align}

For every $k\in \mathbb N$ the system of ODEs \eqref{eq:ODE} has a $T$-periodic solution $b_{k}\in H^1(0,T)$. In particular, for any  $U^{0}\in \operatorname{span}(a_1,\ldots,a_k)$, 
there exist $b_{kj}(s)$ uniquely solving \eqref{eq:ODE} with initial value $b_{kj}(0)=(U^{0},a_j)$, and belonging to $H^1(0,\tilde T)$ for some time $0<\tilde T\leq T$.  If $\tilde T<T$ assume it is maximal--i.e.~$||b_{k}(s)||_{L^2}\to\infty$ as $s\to \tilde T^-$.

Let
\begin{equation} \notag
U_k(y,s)=\sum_{i=1}^k b_{ki}(s)a_i(y).
\end{equation}
We will prove that 
\begin{equation}\label{ineq:uniformink}
||U_k||_{L^\infty (0,T;L^2(\R^3))} + ||U_k||_{L^2(0,T;H^1(\R^3))}<C,
\end{equation}where $C$ is independent of $k$.
Testing the equation against $U_k$ gives the initial estimate
\begin{equation} \label{ineq:1}
\frac 1 2 \frac d {ds} ||U_k||_{L^2}^2 + \frac 1 4 ||U_k||_{L^2}^2+ ||\nabla U_k||_{L^2}^2\leq -(U_k\cdot\nabla (B+W),U_k)  - \langle \mathcal{R}(W), U_k\rangle. 
\end{equation}

We need to estimate the right hand side of \eqref{ineq:1}. 
Note that \eqref{ineq:Wsmall} and the fact that $U_k$ is divergence free guarantee that
\begin{equation}
  \big| ( U_k\cdot \nabla W, U_k )  \big| \leq  \frac 1  {24} ||U_k||_{H^1}^2 .%
\end{equation}
Because $\|B\|_{L^\infty(\R^3\times (0,T))}<\frac 1 {24}$, we have
\[
 |(U_k\cdot\nabla B,U_k)| \leq \frac 1  {24} \| U_k  \|_{H^1}^2.
\]
To estimate the source terms involving $B$ note that since $2<3<2p/(p-2)$ using $p<6$ we have $L^3_w\subset L^2+L^{2p/(p-2)}$, i.e.~we can write $W=W_1+W_2$ where $W_1\in L^2$ and $W_2\in L^{2p/(p-2)}$.  This decomposition of $W$, H\"older's inequality, and the fact that $B\in L^\infty(0,T;L^p)\cap L^\infty(\R^3\times [0,T])$ leads to the bound  
\begin{align}
&\notag \bigg|\int  ( W\cdot\nabla B +B\cdot \nabla W) U_k\,dy\bigg|
\\&\leq C \|\nabla U_k \|_2\big( \|W_1\|_{L^2}\|B\|_{L^\infty} +\|W_2\|_{L^{2p/(p-2)}}\|B\|_{L^p}  \big)\notag
\\&\leq \frac 1  {12} \|\nabla U_k \|_2^2 + C \big( \|W_1\|_{L^2}\|B\|_{L^\infty} +\|W_2\|_{L^{2p/(p-2)}}\|B\|_{L^p}  \big)^2.\label{ineq.sourcetermbound} 
\end{align}
The estimate for the remaining terms from $\langle \mathcal{R}(W), U_k\rangle$ is
\begin{align}\notag
  &|\langle	\partial_s W-\Delta W-\frac 1 2 W-\frac 1 2 y\cdot \nabla W + W\cdot\nabla W , U_k \rangle |
  \\&\leq \frac 1 {24}\|U_k\|_{H^1}^2+C(\|\partial_s W-\Delta W-\frac 1 2 W-\frac 1 2 y\cdot \nabla W\|_{H^{-1}}^2+ \|W\|_{L^4} ).\label{est.RW2}
\end{align}

We thus obtain the inequality
\begin{equation}  \label{ineq:kenergyevolution}
	 \frac d {ds} ||U_k||_{L^2}^2
	 +   \frac 1 4 ||U_k||_{L^2}^2
	 +   \frac 1 4 ||\nabla U_k||_{L^2}^2 \leq C,
\end{equation}for a constant $C$ depending on $W$.
The Gronwall lemma implies
\begin{equation} \label{ineq:gronwall}
\begin{split}
e^{s/4} ||U_k(s)||_{L^2}^2
&\leq ||U^{0}||_{L^2}^2 +  \int_0^{\tilde T} e^{\tau/4}  C \,dt
\\
& \le  ||U^{0}||_{L^2}^2 + e^{T/4} C T,
\end{split}
\end{equation}
for all $s\in [0,\tilde T]$. Note that $\tilde T$ cannot be a blow-up time since the right hand side is finite.  Thus, $\tilde T=T$.  

By \eqref{ineq:gronwall} we can choose $\rho>0$ (independent of $k$) so that 
\begin{equation}\notag
 ||U^{0}||_{L^2}\leq \rho \Rightarrow ||U_{k}(T)||_{L^2}\leq \rho.
\end{equation}
Let $T: B_\rho^k\to B_\rho^k$ map $b_{k}(0)\to b_k(T)$, where $ B_\rho^k$ is the closed ball of radius $\rho$ in $\R^k$.  This map is continuous and thus has a fixed point by the Brouwer fixed-point theorem, implying there exists some $U^{0}\in \operatorname{span}(a_1,\ldots,a_k)$ so that $b_k(0)=b_k(T)$.

It remains to check that \eqref{ineq:uniformink} holds. The $L^\infty L^2$ bound follows from \eqref{ineq:gronwall}
since $\norm{U^0}_{L^2} \le \rho$, which is independent of $k$.  
Integrating  \eqref{ineq:kenergyevolution} in $s \in [0,T]$ and using $U_k(0)=U_k(T)$, we get
\begin{equation} \label{eq2.33}
\int_0^T \big(||U_k||_{L^2}^2
+  ||\nabla U_k||_{L^2}^2 \big)\,dt \le 4 C T
\end{equation}
which gives an upper bound for $\| U_k  \|_{L^2(0,T;H^1 )}$ that is uniform in $k$.

Standard arguments (e.g.~those in \cite{Temam}) imply that there exists a $T$-periodic $U\in {L^2(0,T;H_0^1(\R^3))}$ and a subsequence of $\{U_k\}$ (still denoted by $U_k$) so that 
\begin{align*}
& U_k\rightarrow U  \mbox{~weakly in}~L^2(0,T;X),
\\& U_k\rightarrow U  \mbox{~strongly in}~L^2(0,T;L^2(K))  \mbox{~for all compact sets~}K\subset \R^3,
\\& U_k(s)\rightarrow U (s) \mbox{~weakly in}~L^2 \mbox{~for all}~s\in [0,T].
\end{align*}
The weak convergence guarantees that $U(0)=U(T)$.  Thus $U$ is a periodic weak solution of the perturbed Leray system.  

Let $u=U+W$.  To finish the proof we need to check that \[W-U_0 \in L^\infty(0,T;L^2(\R^3))\cap L^2(0,T;H^1(\R^3)).\]
The $L^\I (0,T;L^2)$ estimate follows from Lemma \ref{lemma:W}.  The $L^2(0,T;H^1)$ estimate is easy to see since $\nabla w\in L^2(0,T;L^2)$ and $\nabla ((1-\xi) U_0)$ is smooth and compactly supported.  
Since $u-W$ and $W-U_0$ are in $L^\infty(0,T;L^2(\R^3))\cap L^2(0,T;H^1(\R^3))$, we also have $u-U_0 \in L^\infty(0,T;L^2(\R^3))\cap L^2(0,T;H^1(\R^3))$.
\end{proof}
 
\section{Construction of a discretely self-similar solution}\label{sec.DSS}

In this section we prove Theorem \ref{theorem.main} on the existence of discretely self-similar solutions.
We first recall a lemma from \cite{BT1}. 
\begin{lemma}\label{th:2.1}
Suppose $a_0$ is $\la$-DSS, divergence free, and belongs to $L^3_w$. Let $x,t,y,s$ satisfy \eqref{variables}.  Then
\begin{equation}\label{def:U0} 
U_0(y,s)= {\sqrt {t}} (e^{t\Delta}a_0)(x), 
\end{equation}satisfies Assumption \ref{AU_0} with $T=2\log \la$ and any $q \in (3,\I]$.
\end{lemma}

We are now ready to prove Theorem \ref{theorem.main}.

\begin{proof}[Proof of Theorem \ref{theorem.main}]    

Assume $3< p< 6$.
We seek a solution $v$ to 3D NSE for a given divergence free, $\la$-DSS initial data $v_0\in \Bp$ by considering a perturbed problem.  Assume $v_0$ is $\la$-DSS. By Lemma \ref{lemma.profileslicing}, we can decompose $v_0=a_0+b_0$ where $a_0$ and $b_0$ are both $\la$-DSS, $a_0\in L^3_w$ and $\|b_0\|_{\Bp}<\epsilon_0$, where $\epsilon_0$ is a small constant.

{By \cite[Theorem 5.27]{BCD}, if $\epsilon_0$ is sufficently small, there is a unique solution $b\in K_p(\I)$ of 3D NSE with initial data $b_0$, where 
\EQ{
K_p(\I)  = \bket{ u \in C((0,\infty);L^p)\ : \ \norm{u}_{K_p} = \sup_{0<t<\infty} t^{\frac 12 - \frac 3{2p}}\norm{u(t)}_{L^p}}.
} 
By \cite[Theorem 5.40]{BCD}, $b$ also belongs to a strict subspace $E_p$ of $L^\infty(\R^+; \Bp)$, but we do not need this fact here.

 Let $b$ be the above solution and}
 $\pi_b$ the corresponding pressure.   Then, $v=a+b$ is a solution to 3D NSE with pressure $\pi=\pi_a+\pi_b$ if and only if $(a,\pi_a)$ satisfies
\begin{align}\label{eq.a}
&a_t-\Delta a +a\cdot \nabla a+b\cdot \nabla a+a\cdot \nabla b +\nabla \pi_a=0
\\\notag &\nabla \cdot a = 0;\qquad a(x,0)=a_0(x).
\end{align}

Note that $b$ is $\la$-DSS by the uniqueness of small solutions in the Koch-Tataru class. Therefore, $\sqrt{t}b(x,t)=B(y,s)$ where $B$ is time periodic with period $T=2\log\la$.   Also, $b$ is smooth (see \cite{LR}) and, therefore, so is $B$. By \cite[Theorem 20.3]{LR} (see also \cite{BCD}) we have
\[
\|b(t)\|_{L^p}^2\lesssim t^{-1+3/p} \|b_0 \|_{\Bp}^2,
\]
and, therefore, $B(y,s)\in L^\infty(\R ; L^p (
\R^3))$.  Indeed, we also have 
Since $b$ is in the Koch-Tataru class we also have decay in $L^\infty$, i.e.,~\[\|b(t)\|_{L^\infty(\R^3)}\lesssim t^{-1/2}\|b_0\|_{BMO^{-1}}. \]
Provided $\epsilon_0$ is sufficiently small (it can be chosen to be arbitrarily small in Lemma \ref{lemma.profileslicing}) it follows that 
\[
\|  B(y,s)\|_{L^\infty(\R\times \R^3)}<\frac 1 {24}.
\]

Let $U_0(y,s)= {\sqrt {t}} (e^{t\Delta}a_0)(x)$, as in \eqref{def:U0}.
Because $a_0$ is $\la$-DSS, divergence free, and belongs to $L^3_w$, we have by Lemma \ref{th:2.1} that $U_0(y,s)$ satisfies Assumption \ref{AU_0}.  

By Theorem \ref{thrm:existenceOnR3} with $B$ and $U_0$, we obtain a $T$-period solution $u$ to \eqref{eq:wholeSpaceLeray} and, undoing the DSS transform, we recover a $\la$-DSS solution $a$ to \eqref{eq.a}.  We thus obtain the desired $\la$-DSS solution $v=a+b$ to 3D NSE.

The pressure distribution $\pi$ for $v$ is given by $\pi=\pi_a+\pi_b$ where $\pi_a$ is the image under the change of variables \eqref{variables} of the pressure distribution $p(y,s)$ for $u(y,s)$ and $\pi_b$ is the pressure distribution associated with that Koch-Tataru solution $b$.

To complete the proof, note that  
\begin{equation} \notag  a-e^{t\Delta}a_0 \in L^\infty(1,\la^2;L^2(\R^3))\cap L^2(1,\la^2;H^1(\R^3)). \end{equation}
The $\la$-DSS scaling property implies that for all $t\in (0,\I)$ that
\begin{equation}\label{ineq.a}
||a(t)-e^{t\Delta}a_0||_{L^2}^2\lesssim t^{1/2} \sup_{1\leq \tau\leq \la^2} ||a(\tau)-e^{\tau \Delta}a_0||_{L^2}^2,
\end{equation}  
and
\begin{equation}\label{ineq.a2}
\int_0^t \| \nabla (a(\tau)-e^{\tau\Delta}a_0) \|_{L^2}^2\,d\tau \lesssim \int_1^{\la^2} \| \nabla (a(\tau)-e^{\tau\Delta}a_0) \|_{L^2}^2\,d\tau.
\end{equation} 
So, for any $t>0$, by interpolating  between \eqref{ineq.a} and \eqref{ineq.a2}, we see that
\EQ{
\int_{0}^t \norm{a(\tau)-e^{\tau\Delta}a_0}_{L^r}^q d\tau  &\le C_r t^{q/4},
}
for {all $r\in (2,6]$} and $q$ such that $\frac 2 q +\frac 3 r=\frac 3 2$.  This proves \eqref{th1.3-1}.

{
We found
$b\in K_p(\I)$ by \cite[Theorem 5.27]{BCD}. Its proof uses
\cite[Lemma 5.29]{BCD}, which implies that, for $ \frac 2p -\frac 13 <\frac 1r \le \frac 2p$,
\[
\norm{b(t)-e^{t\Delta} b_0}_{L^r} \le C \norm{b_0}_{\Bp}^2 t^{-\frac 12 + \frac 3{2r}},\quad \forall t>0.
\]
This shows \eqref{th1.3-2}.
Since $p \in (3,6)$, $\frac 13 < \frac 2p < \frac 23$, we can choose 
$\frac 1r \in (\frac 13, \frac 2p]$, i.e.,~$r \in [\frac p2, 3)$.
Then the exponent $-\frac 12 + \frac 3{2r}>0$ and $\norm{b(t)-e^{t\Delta} b_0}_{L^r} \to 0$ as $t \to 0^+$.
}
\end{proof}

\section{Self-similar solutions}\label{sec.SS}

In this section we prove Theorem \ref{theorem.main2} on the existence of self-similar solutions.

We first decompose the initial data.  
The definition of Besov spaces given in \S2 can be extended to describe non-homogeneous Besov spaces on compact 
smooth manifolds as in \cite[Ch.~23]{LR}. Let $M$ be a compact smooth manifold of dimension $d$ and assume $T$ is a distribution on $M$.  Then $T\in B^s_{p,q}(M)$ if and only if for every open subset $\Om$ of $M$, smooth differomorphism $h:\Om \to \R^d$, and test function $\phi$ supported on $\Om$, we have $ (\phi T )\circ h^{-1}\in B^s_{p,q}(\R^d)$. 
The norm of $B^s_{p,q}(M)$ is defined using a finite atlas of $M$ (the choice of atlas does not matter; any two give equivalent definitions).  Let $A$ be a finite set.
Let $\{\Om_\al\}_{\al\in A}$ be an open cover of $M$.  Let $h_{\al}$ be a diffeomorphism from $\Om_\al$ to $\R^2$. Let $\{\phi_\al\}_{\al\in A}$ be a partition of unity of $M$ with $\supp \phi_{\al} \subset \Om_{\al}$.  
Then,  \[\|T \|_{B_{p,p}^{3/p-1}(S^2)}:= \sum_{\al\in A} \|(\phi_\al T)\circ h^{-1}_\al  \|_{B_{p,p}^{3/p-1}(\R^2)}.\] 
Furthermore we have 
\[ T=\sum_{\al\in A,j\geq -1} T^\al_j,\]
where
\[
T_j^\al =  (\Delta_j (\phi_\al T \circ h_\al^{-1}))\circ h_\al.
\]  
  
 We need a lemma due to Cannone (see \cite[Proposition 23.1]{LR}).  
 Here, $S^{2}$ denotes the unit sphere in $\R^3$.
  
\begin{lemma}[Cannone's Lemma]\label{lemma.cannone} Let $p\in [1,\I]$ and $T$ be a distribution on $\R^3$ which is homogeneous of degree $-1$. The following are equivalent:
\begin{itemize}
\item[A.]$T\in \Bp(\R^3)$
\item[B.]$T|_{S^2} \in B_{p,p}^{3/p-1}(S^{2})$.
\end{itemize}
\end{lemma}
By inspecting the proof of Lemma \ref{lemma.cannone} (\cite[pg. 238-239]{LR}) it is clear that 
\begin{align}\label{ineq.cannone}
  \|T\|_{\Bp(\R^3)}\leq \kappa  \|T|_{S^2}\|_{B_{p,p}^{3/p-1}(S^2)},
\end{align}
for a constant $\kappa$ that does not depend on $T$.
This leads to an analogue of Lemma \ref{lemma.profileslicing} for self-similar functions.
\begin{lemma} \label{lemma.profileslicing2}
Let $f$ be divergence free, $-1$-homogeneous, and belong to $\dot B_{p,\infty}^{3/p-1}$ for some $p\in (3,\infty)$. For any $\epsilon>0$, there exist divergence free $-1$-homogeneous distributions $a\in L^3_w$ and $b\in \Bp$ so that 
 $f = a + b$ and $ \|b  \|_{\dot B_{p,\infty}^{3/p-1}}<\epsilon$. 
\end{lemma}
\begin{proof}Let $A$ be a finite set.
Let $\{\Om_\al\}_{\al\in A}$ be an open cover of $S^2$.  Let $h_{\al}$ be a diffeomorphism from $\Om_\al$ to $\R^2$. Let $\{\phi_\al\}_{\al\in A}$ be a partition of unity of $S^2$ with $\supp \phi_{\al} \subset \Om_{\al}$.  
Then, $f|_{S^2}=\sum_{\al\in A,j\geq -1} \Delta_j (\phi_\al f)$ and $\|f|_{S^2}\|_{B_{p,p}^{3/p-1}(S^2)}\equiv \sum_{\al\in A} \|\phi_\al f\circ h^{-1}_\al  \|_{B_{p,p}^{3/p-1}(\R^2)}$.

Choose $J$ so that, letting 
\[
b_0=\sum_{\al\in A,j\geq J}  (\Delta_j (\phi_\al f \circ h_\al^{-1}))\circ h_\al,
\] 
we have $\| b_0\|_{B_{p,p}^{3/p-1}(S^2)}<\epsilon / \kappa$ (this is possible since the summation index is finite).  Let $a_0=f|_{S^2}-b_0$.
Extend  $a_0$ and $b_0$ to $a$ and $b$ by the $-1$-homogeneous scaling relationship. By \cite[Lemma 23.2]{LR}, $a+b=f$.  By \eqref{ineq.cannone},
\[\|b\|_{\Bp}\leq \kappa \|b_0\|_{B_{p,p}^{3/p-1}}\leq \epsilon.
\]
Furthermore, $a_0\in L^\I (S^2)$ and, therefore, $a\in L^3_w$. To conclude re-define $a$ and $b$ after applying the divergence free projector to each field as in the proof of Lemma \ref{lemma.profileslicing}.
\end{proof}

\begin{proof}[Proof of Theorem \ref{theorem.main2}] Assume $v_0$ is as in the statement of the theorem.  By Lemma \ref{lemma.profileslicing2} we can write $v_0=a_0+b_0$ where $a_0$ and $b_0$ are $-1$-homogeneous, $a_0\in L^3_w$, and $\|b_0\|_{\Bp}$ is smaller than the Koch-Tataru constant. 
Let $b$ be the self-similar Koch-Tataru solution evolving from $b_0$ with pressure $\pi_b$ and let $B(y)=b(x,1)$ under the self-similar change of variables.  To find a self-similar solution $v$ to \eqref{eq:NSE} with initial data $v_0$, we find a solution $a$ to the perturbed problem \eqref{eq.a}.  The corresponding self-similar profile $A$ is divergence free and satisfies the perturbed Leray equation
\[
-\Delta A -\frac 1 2 A-\frac 1 2 y\cdot\nabla A+A\cdot \nabla A +B\cdot\nabla A + A\cdot \nabla B +\nabla p = 0.
\]
Let $A_0$ be the solution to the heat equation with initial data $a_0$ and let $U_0(y)=A_0(x,1)$ under the self-similar change of variables \eqref{variables}.   Then, $U_0$ satisfies Assumption \ref{AU_0} for any $T>0$ by Lemma \ref{th:2.1}. Applying Lemma \ref{lemma:W} for $U_0$, $q=10/3$, and $\al=1/24$, gives a small asymptotic profile $W$.  If $A=U+W$, then $U$ is divergence free and satisfies 
\[
-\Delta U -\frac 1 2 U-\frac 1 2 y\cdot\nabla U +(U+W)\cdot\nabla U  + U\cdot\nabla B +B\cdot\nabla U +U\cdot \nabla W  +\nabla p=-\mathcal R (W),
\]
where
\[
\mathcal R (W)=-\Delta W-\frac 1 2 W-\frac 1 2 y\cdot \nabla W + W\cdot\nabla W+B\cdot\nabla W+W\cdot\nabla B.
\]
We now construct such a $U$ using a Galerkin scheme.
Let $\{ a_k \}\subset \mathcal V$ be an orthonormal basis of $H$.  For $k \in \mathbb N$, the approximating solution \[
U_k(y)=\sum_{i=1}^k b_{ki}a_i(y),
\] is required to satisfy
\begin{align}\label{eq:stationaryODE}
 & \sum_{i=1}^k A_{ij}b_{ki} +\sum_{i,l=1}^k B_{ilj} b_{ki}b_{kl} +C_j=0,%
\end{align}
for $j\in \{1,\ldots,k\}$,
where
\begin{align}
\notag A_{ij}&=- (\nabla a_{i},\nabla a_j) 
		+ \frac 1 2(a_i+y\cdot \nabla a_i, a_j) 
		 -(    a_i\cdot \nabla (W+B),a_j)
		- ((W+B)\cdot\nabla a_i, a_j)
\\\notag B_{ilj}&=- (a_i \cdot\nabla a_l, a_j)
\\\notag C_j&=-\langle \mathcal R (W),a_j\rangle.
\end{align}
Let $P(x):\R^k\to\R^k$ denote the mapping  
\[
P(x)_j=\sum_{i=1}^k A_{ij}x_{i} +\sum_{i,l=1}^k B_{ilj} x_{i}x_{l} +C_j.
\]
For $x\in \R^k$, let $\xi=\sum_{j=1}^k x_j a_j$.  We have 
\EQ{
\label{eq4.6}
P(x)\cdot x &= -\frac 1 4 ||\xi||_{L^2}^2
	 -   ||\nabla \xi||_{L^2}^2 +(\xi \cdot \nabla \xi,W+B) - \bka{\cR(W),\xi}
\\
& \le  -\frac 1 8 ||\xi||_{L^2}^2  -   \frac 1 2 ||\nabla \xi||_{L^2}^2 +C_*^2 \norm{\cR(W)}_{H^{-1}}^2
\\
& \le  -\frac 1 8 |x|^2 + C_*^2 \norm{\cR(W)}_{H^{-1}}^2,
}
using the smallness of $\norm{W}_{L^\I}$ and $\norm{B}_{L^\I}$, as well as the estimates \eqref{ineq.sourcetermbound} and \eqref{est.RW2} for $\mathcal R(W)$ in Section \ref{sec.DSS}. We conclude that
\[
P(x)\cdot x< 0,\quad \text{if } |x|=\rho := 4C_* \norm{\cR(W)}_{H^{-1}}.
\]
By Brouwer's fixed point theorem, there is one $x$ with $|x|<\rho$ such that $P(x)=0$. Then $U_k=\xi$ is our approximation solution satisfying
\eqref{eq:stationaryODE}. 
By the first inequality of \eqref{eq4.6} and $P(x)=0$, $U_k$ also satisfies
the  \emph{a priori} bound
\[
 \norm{U_k}_{L^2}^2 + \norm{\nabla U_k}_{L^2} ^2\le 8C_*^2 \norm{\cR(W)}_{H^{-1}}^2.
\]
This bound is sufficient to find a subsequence with a weak limit in $H^1(\R^3)$ and a strong limit in $L^2(K)$ for any compact set $K$ in $\R^3$ -- i.e.~there exists a solution $U$ with $U\in H^1(\R^3)$. We now obtain $A$ by setting $A=U+W$.  Note that $A\in H^1_{\mathrm{loc}}\cap L^q$ for $3< q\leq  6$,  
and, following \cite[pp. 287-288]{NRS} or \cite[pp. 33-34]{Tsai-ARMA}, if we define
\[
p = \sum_{i,j} R_i R_j (A_i A_j),
\]
where $R_i$ denote the Riesz transforms,
then $(A,p)$ solve the perturbed stationary Leray system in the distributional sense.
To obtain a solution to \eqref{eq:NSE}, pass from the self-similar profile $A$ to the field $a$ at time $t=1$ using the change of variable \eqref{variables} and extend $a$ to all times using the ansatz \eqref{ansatz1}.  Also do this for the pressure; let $\pi_a$ be self-similar extension of the image of $p$ under the change of variables \eqref{variables}. Finally, let $v=a+b$ and $\pi=\pi_a+\pi_b$.   
\end{proof}

\section{Relationships between function spaces}\label{Appendix}

In this section we state and prove lemmas clarifying the relationships between several function spaces. The first two lemmas give examples of $2$-DSS vector fields in $\Bp$ that are not in other spaces. They ensure that Theorem \ref{theorem.main} is new in comparison to Theorem \ref{thrm.old},  \cite[Theorem 16.3]{LR2}, and \cite{Chae-Wolf}.

\begin{lemma}\label{lemma.strictembedding}For any $p,q\in (3,\infty)$ with $q<p$, there exists a $2$-DSS function $f$ belonging to $\Bp \setminus \dot B_{q,\infty}^{3/q-1}$.  In particular $f\in \Bp\setminus L^3_w$. 
\end{lemma}

\begin{lemma}
\label{lemma.notsquareintegrable}There exists a $2$-DSS vector field in $\Bp\setminus L^2_{\mathrm{loc}}$ whenever $p>3$.
\end{lemma}

The last lemma is included for illustrative purposes.

\begin{lemma}\label{lemma.BMOinversenotBesov}
There exists a $2$-DSS vector field in $BMO^{-1}\setminus \Bp$ whenever $0<p<\I$.
\end{lemma}

Each of these lemmas is proved by constructing explicit examples starting with a wavelet basis. We recall the essentials about wavelets.  
Meyer constructed wavelets in \cite[p. 108]{Meyer}.  In particular, there exists a family of functions $\{\psi_{\e,j,k}\}_{\e=1,\ldots,7;j\in \Z;k\in \Z^3}$ so that
\begin{enumerate}
\item they are generated from given functions $\psi_\e$ for $\e=1,\ldots,7$ by
\[
	\psi_{\e,j,k}(x)=2^{3j/2}\psi_\e(2^j x-k),
\]
\item they constitute an orthonormal basis of $L^2(\R^3)$,
\item they are compactly supported in dyadic cubes, in particular, for $k=(k_1,k_2,k_3)$,
\[
\operatorname{supp}	\psi_{\e,j,k} \subset  \big[2^{-j}k_1, 2^{-j}(k_1+1)\big]\times \big[2^{-j}k_2, 2^{-j}(k_2+1)\big]\times \big[2^{-j}k_3, 2^{-j}(k_3+1)\big].
\]
\end{enumerate} 
Moreover the wavelets can be taken with arbitrarily high regularity, with enlarged compact support.
The parameter $\e$ plays no role in what follows and is consequently suppressed.

Assume $1\leq p\leq \infty$ and $f$ is a distribution given by
\begin{equation}\label{wavelet-series}
 f = \sum_{j,k}\alpha_{j,k}\psi_{j,k},
\end{equation}
with convergence understood in the space of tempered distributions $\mathcal S'$.
Then, $f\in \Bp$ if and only if
\[
\|f\|_{\bp }:=\sup_{j\in \Z} 2^{j/2} \bigg( \sum_{k} |\alpha_{j,k}|^p   \bigg)^{1/p}<\infty,
\]for some sequence of wavelet coefficients $\alpha_{j,k}$ (see \cite[Proposition 6]{Cannone-handbook} and \cite[p.~200]{Meyer}), and, moreover,
\begin{equation}\label{norm-equiv}
 \|f\|_{\Bp} \sim \|f\|_{\bp}.
\end{equation}
For  $f\in \Bp$, the coefficients in the series \eqref{wavelet-series} are uniquely determined since $\alpha_{j,k} = \langle \psi_{j,k},f \rangle$.

Our first lemma describes the relationship between different scales in a discretely self-similar function.  This is essentially a wavelet version of the relationship $ \dot \Delta_j f(x)=2^{j-i}\dot \Delta_i f(2^{j-i} x)$ for every $i,j\in \Z$, which we saw in Section \ref{sec.technical}.

\begin{lemma}\label{lemma.scaletoscale}
Let $f$ be a tempered distribution and let 
\[
f_j=\sum_{k\in \Z^3} \alpha_{j,k}\psi_{j,k},
\]
where $\{\psi_{j,k}\}$ is a $2$-regular wavelet basis and $\alpha_{j,k}=\langle \psi_{j,k},f\rangle$ for all $j\in \Z$ and $k\in \Z^3$ so that
$f=\sum_{j\in \Z}f_j$.  The following are equivalent:
\begin{itemize}
\item[i.] $f$ is $2$-DSS,
\item[ii.]	$ f_j(x)=2^{j-i} f_i(2^{j-i} x)$ for every $i,j\in \Z$,
\item[iii.] $ \alpha_{j,k} = 2^{-(j-i)/2} \alpha_{i,k}$ for every $i,j\in \Z$.
\end{itemize} 
\end{lemma}

\begin{proof}[Proof of Lemma \ref{lemma.scaletoscale}] Note that
\begin{align}\label{psijtoi}\psi_{j,k}(x)
&=2^{3j/2}\psi(2^jx-k)
=2^{3j/2}\psi( 2^i 2^{j-i}x-k)
=2 ^{3(j-i)/2}\psi_{i,k}(2^{j-i}x).
\end{align} 

(\emph{i.$\implies$iii.}) Assume $f$ is $2$-DSS and let $i,j\in \Z$.  By the uniqueness of wavelet coefficients and \eqref{psijtoi}  we have
\[\alpha_{j,k} 
= \int \psi_{j,k}(y) f(y)\,dy
= 2^{3(j-i)/2} \int \psi_{i,k}(2^{j-i}y)f(y)\,dy.
\]
Since $f$ is $2$-DSS we have
\EQN{
 \int \psi_{i,k}(2^{j-i}y)f(y)\,dy
& = 
 \int \psi_{i,k}(2^{j-i}y)2^{j-i}f(2^{j-i}y)\,dy
\\
& = 2^{-2(j-i)} \int \psi_{i,k}(z)f(z)\,dz
 = 2^{-2(j-i)} \al_{i,k},
}
where we have set $z=2^{j-i}y$.
Therefore,
\[
\al_{j,k}=2^{-(j-i)/2} \al_{i,k}.
\]

(\emph{iii.$\implies$ii.}) Assume $\al_{j,k}=2^{-(j-i)/2} \al_{i,k}$ for all $i,j\in \Z$.  Then, 
\begin{align*}
  f_j(x)&=\sum_{k\in\Z^3} \al_{j,k}\psi_{j,k}(x)
  =\sum_{k\in \Z^3} 2^{-(j-i)/2} \al_{i,k} 2^{3(j-i)/2}\psi_{i,k}(2^{j-i}x)=
2^{j-i} f_i(2^{j-i} x)  ,
\end{align*}
where we have used \eqref{psijtoi}.

(\emph{ii.$\implies$i.}) Assume $ f_j(x)=2^{j-i} f_i(2^{j-i} x)$ for every $i,j\in \Z$.  Fix $j\in \Z$ and let $i=j+1$.  Then
\[
f_j(2 x) = 2^{j-i}f_i(  2^{j-i+1} x)=2^{-1} f_i(x).
\]
Then,
\[2 f(2 x)=  2 \sum_{j\in \Z} f_j(2 x) = \sum_{i\in \Z} f_i(x)=f(x),
\]
implying $f$ is $2$-DSS.
\end{proof}

\begin{proof}[Proof of Lemma \ref{lemma.strictembedding}] Assume $q\in (3,\I]$.
For $n\in \N$, let $\hat n=(n,0,0)$. Let 
\[ f_0= \sum_{n\in \N} n^{-1/q} \psi_{0,\hat n}.\] 
Let $f_j(x)=2^j f_0(2^j x)$ and let $f(x)=\sum_j f_j(x)$.  Then, $f$ is $2$-DSS by Lemma \ref{lemma.scaletoscale}.  Also by Lemma \ref{lemma.scaletoscale} we have
\[
2^{j/2}\bigg(\sum_{n\in \N} |\al_{j,\hat n}|^p\bigg)^{1/p} =\bigg(\sum_{n\in \N} |\al_{0,\hat n}|^p\bigg)^{1/p}=\bigg(\sum_{n\in \N} n^{-p/q}\bigg)^{1/p}.
\]
If $p>q$, then $f\in \Bp$.  If $p=q$ then the above series diverges.  Thus $f\notin \dot B^{3/q-1}_{q,\infty}$ and, since $L^3_w\subset \dot B_{q,\infty}^{3/q-1}$, $f \notin L^3_w$.
\end{proof}

\begin{proof}[Proof of Lemma \ref{lemma.notsquareintegrable}] 
As in the proof of Lemma \ref{lemma.strictembedding}, we first construct $f_0$ and then extend it to a $2$-DSS vector field using Lemma \ref{lemma.scaletoscale}.    If $|k|\geq 2$ then let $\al_{0,k}=|k|^{-1}$.  Let $\al_{0,k}=0$ for $|k|<2$.   Define $f$ using Lemma \ref{lemma.scaletoscale}.
Then $f\in \Bp$ because $\{ \al_{0,k} \}\in l^p(\Z^3)$ provided $p>3$.

It remains to show that $f\notin L^2_{\mathrm{loc}}$.  Let $A_1=\{x:1\leq |x|\leq 2   \}$.   Let $\phi\in C^\I$ be non-negative, supported on $A_1^*=\{x:2^{-1}\leq |x| \leq 4 \}$, and equal $1$ on $A_1$.  Let $S_j=\{ k: \supp \psi_{j,k}\subset A_1 \}$.  Note that $|S_j|\sim 2^{3j}$  {for $j \gg 1$}.  If
\[
\phi f = \sum \be_{j,k}\psi_{j,k},
\]
then $\be_{j,k}=\al_{j,k}$ whenever $\supp \psi_{j,k}\subset A_1$.  Since we are working with an orthonormal basis we have
\[
\int_{\R^3} (\phi f)^2\,dx = \sum_{j\in \Z} \sum_{k\in \Z^3} |\be_{j,k}|^2 \geq \sum_{j\in \N} \sum_{k\in S_j} |\al_{j,k}|^2.
\]
Note that if $k\in S_j$ then $|k|\sim 2^{j}$. So, $\al_{0,k}=|k|^{-1}\sim 2^{-j}$ for all $k\in S_j$. Using Lemma \ref{lemma.scaletoscale} we have
\begin{align*} \sum_{j\in \N}\sum_{k\in S_j} \al_{j,k}^2
= \sum_{j\in \N} \sum_{k\in S_j}   2^{-j} \al_{0,k}^2
\sim  \sum_{j\in \N} 2^{3j}  2^{-j} 2^{-2j}=\I.
\end{align*} 
Hence $\phi f\notin L^2_{\mathrm{loc}}$ and, since $|\phi f|\leq |f|$, neither is $f$.
\end{proof}
 
\begin{remark}
More can be said, in particular the function $f$ constructed above does not belong to $L^q(A_1)$ for any $q\in (1,\I)$.  This is clear when $q\in (2,\I)$ by H\"olders inequality.  For $q\in (1,2)$ we can use the fact that $L^{q}$ embeds continuously in $\dot B^{0}_{q,2}$ (see \cite[Theorem 2.40]{BCD}) and adapt the above argument to show that $\phi f\notin \dot B^{0}_{q,2}$, i.e.
\[
 \sum_{j\in \Z} \bigg(  2^{(3/2-3/q)j} \bigg(\sum_{k\in \Z^3} |\be_{j,k}|^q  \bigg)^{\frac 1 q} \bigg)^{2} =\I.
 \]
\end{remark}
 
\begin{proof}[Proof of Lemma \ref{lemma.BMOinversenotBesov}]

We construct a $2$-DSS vector field $f$ which belongs to $BMO^{-1}\setminus \Bp$.  This field is similar to the one discussed in remark (4) following \cite[Theorem 1.2]{BT1}. 
Let $A_j$ equal the collection of $k\in \Z^3$ so that the cube $Q_{j,k}$ is touching the point $k_j=(3(2^{j-1}),0,0)$.  Then,
\begin{align*}
A_j=\{&(3(2^{j-1}) ,0,0 )  ,(3(2^{j-1}) ,-1, 0) ,
\\&(3(2^{j-1}) ,0,-1 ) ,(3(2^{j-1}) ,-1, -1) ,
\\&(3(2^{j-1})-1 ,0,0 )  ,(3(2^{j-1})-1 ,-1, 0) ,
\\&(3(2^{j-1}) -1,0,-1 ) ,(3(2^{j-1})-1 ,-1, -1) \}.
\end{align*}
For all $k\in A_j$ let $\al_{j,k}=2^{-(j-1)/2}$ and let $f$ be the $2$-DSS extension of 
\[
\sum_{j\in\N}\sum_{k\in A_j} \al_{j,k}\psi_{j,k}.
\]
Let $f_j=\sum_{k\in \Z^3} \al_{j,k} \psi_{j,k}$. Then, $\supp f_j \subset \R^3\setminus B_{2^{-j}}(0)$ and $f_j$ repeats along the positive $x_1$-axis. Hence $f_j\in L^\I\setminus L^p$ for all $p$ and, since $f$ is $2$-DSS,  $f\in \dot B_{\I,\I}^{-1}\setminus \Bp$ for all $p<\I$.  The function $f$ is singular at the points $k_j$ and each singularity is of order $|x|^{-1}$.

With a little work we can also show that $f\in BMO^{-1}$. 
Recall 
\[
\|f\|_{BMO^{-1}} = \sup_{Q} \frac 1 {|Q|} \sum_{Q_{j,k}\subset Q} (2^{-j}|\al_{j,k}|)^2.
\]
Since $f$ is $2$-DSS we have by Lemma \eqref{lemma.scaletoscale} that if $|Q|\sim 2^{-3J}$, then
\begin{align*}
\sup_{Q} \frac 1 {|Q|} \sum_{Q_{j,k}\subset Q} (2^{-j}|\al_{j,k}|)^2
&=\frac 1 {2^{-3J}}\sum_{Q_{j,k}\subset Q} 2^{-2j}|\al_{j,k}|^2
\\&=\frac 1 {2^{-3J}}\sum_{Q_{i,k}\subset Q_0} 2^{-2(i+J)}|\al_{i+J,k}|^2
\\&=\frac 1 {2^{-3J}}\sum_{Q_{i,k}\subset Q_0} 2^{-3J}2^{-2i}|\al_{i,k}|^2,
\end{align*}
where $|Q_0|\sim 1$.   Thus the $BMO^{-1}$ norm is determined by taking the supremum over cubes of volume $\sim 1$.  The worst case scenario for such cubes is finite by our definition of the wavelet coefficients of $f$. Therefore, $f\in BMO^{-1}$.

\end{proof}

\section*{Acknowledgments}
The research of both authors was partially supported by the NSERC grant 261356-13 (Canada). That of Z.B. was also partially supported by the NSERC grant 251124-12. {We thank Dr.~Tong-Keun Chang for finding an error in a previous proof of Lemma \ref{lemma.profileslicing}. }

Zachary Bradshaw,  Department of Mathematics, University of British
Columbia, Vancouver, BC V6T 1Z2, Canada;
e-mail: zbradshaw@math.ubc.ca

\medskip

Tai-Peng Tsai, Department of Mathematics, University of British
Columbia, Vancouver, BC V6T 1Z2, Canada;
e-mail: ttsai@math.ubc.ca

\end{document}